\documentclass[12pt]{article}
\usepackage{amssymb}
\usepackage{amsthm}
\usepackage{amsmath}
\usepackage{graphicx}
\usepackage{enumerate}

\DeclareMathOperator{\DM}{DM}
\DeclareMathOperator{\BDM}{\mathbf{DM}}

\newtheorem{theorem}{Theorem}[section]
\newtheorem{definition}[theorem]{Definition}
\newtheorem{lemma}[theorem]{Lemma}

\newtheorem{example}[theorem]{Example}

\title{Consistent posets}
\author{Ivan~Chajda and Helmut~L\"anger$^1$}
\date{}
\begin{document}
\footnotetext[1]{Corresponding author}
\footnotetext[2]{Support of the research of the authors by the Austrian Science Fund (FWF), project I~4579-N, and the Czech Science Foundation (GA\v CR), project 20-09869L, entitled ``The many facets of orthomodularity'', as well as by \"OAD, project CZ~02/2019, entitled ``Function algebras and ordered structures related to logic and data fusion'', and, concerning the first author, by IGA, project P\v rF~2020~014, is gratefully acknowledged.}
\maketitle
\begin{abstract}
We introduce so-called consistent posets which are bounded posets with an antitone involution $'$ where the lower cones of $x,x'$ and of $y,y'$ coincide provided $x,y$ are different form $0,1$ and, moreover, if $x,y$ are different form $0$ then their lower cone is different form $0$, too. We show that these posets can be represented by means of commutative meet-directoids with an antitone involution satisfying certain identities and implications. In the case of a finite distributive or strongly modular consistent poset, this poset can be converted into a residuated structure and hence it can serve as an algebraic semantics of a certain non-classical logic with unsharp conjunction and implication. Finally we show that the Dedekind-MacNeille completion of a consistent poset is a consistent lattice, i.e.\ a bounded lattice with an antitone involution satisfying the above mentioned properties.
\end{abstract}



\section{Introduction}
In some non-classical logics the contraposition law is assumed. An algebraic semantics of such logics is provided by means of De Morgan posets, i.e.\ bounded posets equipped with a unary operation $'$ which is an antitone involution. This operation $'$ is then considered as a negation. Clearly, $0'=1$ and $1'=0$, but we do not ask $'$ to be a complementation. In particular, this is the case of the logic of quantum mechanics represented by means of an orthomodular lattice or an orhomodular poset in a broad sense. In orthomodular lattices the following implication holds
\[
x\leq y\text{ and }y\wedge x'=0\text{ imply }x=y.
\]
In fact, for an ortholattice this condition is necessary and sufficient for being orthomodular. When working with orthomodular posets, the aforementioned condition can be expressed in the form
\[
x\leq y\text{ and }L(y,x')=\{0\}\text{ imply }x=y
\]
where $L(y,x')$ denotes the lower cone of $y$ and $x'$.

However, there are logics where such a condition can be recognized as too restrictive. Hence, we can relax the equality $x=y$ by asking that $x,y$ have the same lower cones generated by the pairs including the involutive members, i.e.\ we consider the condition
\[
x\leq y\text{ and }L(y,x')=\{0\}\text{ imply }L(x,x')=L(y,y').
\]
Of course, if $\mathbf P=(P,\leq,{}',0,1)$ is a bounded poset where the operation $'$ is a complementation then
\[
L(x,x')=\{0\}=L(y,y')
\]
for all $x,y\in P$. However, this is rather restrictive. Hence, we do not ask in general that $'$ is a complementation, but $\mathbf P$ should satisfy $L(x,x')=L(y,y')$ for $x,y\neq0,1$.

Starting with this condition, we can release the assumption that $x,y$ are comparable but, on the other hand, we will ask that $L(x,y)=\{0\}$ if and only if at least one of the entries $x,y$ is equal to $0$. Such a {\em poset} will be called {\em consistent} in the sequel. It represents certain logics satisfying De Morgan's laws. Usually, a logic is considered to be well-founded if it contains a logical connective implication which is related with conjunction via the so-called adjointness. In what follows, we show that consistent posets can be represented by means of algebras (with everywhere defined operations) which enables to use algebraic tools for investigating these posets. Moreover, we show when these posets can be organized into a kind of residuated structure, i.e.\ we introduce conjunction and implication related via adjointness. Of course, working with posets, one cannot expect that these logical connectives will be operations giving a unique result for given entries. We will define operators assigning to the couple $x,y$ of entries a certain subset of $P$. It is in accordance with the description of uncertainty of such a logic based on the fact that a poset instead of a lattice is used.

\section{Preliminaries}

In our previous papers \cite{CL14} and \cite{CL18} we studied complemented posets. We showed when such a poset can be represented by a commutative directoid (\cite{CKL1}, \cite{CL11} and \cite{JQ}) and when it can be organized into a residuated or left-residuated structure (\cite{CKL2}, \cite{CL14}, \cite{CL17}, \cite{CL18} and \cite{CL1}). Now we introduce a bit more general posets with an antitone involution which need not be a complementation but it still shares similar properties. We again try to characterize these posets by identities or implications of corresponding commutative meet-directoids similarly as it was done in \cite{CKL1}. This approach has the advantage that commutative directoids are algebras similar to semilattices and hence we can use standard algebraic tools for their constructions, see e.g.\ \cite{JQ}. We also solve the problem when these so-called consistent posets can be converted into residuated or left-residuated structures.

For the reader's convenience, we recall several concepts concerning posets.

Let $\mathbf P=(P,\leq)$ be a poset, $a,b\in P$ and $A,B\subseteq P$. We write $a\parallel b$ if $a$ and $b$ are incomparable and we extend $\leq$ to subsets by defining
\[
A\leq B\text{ if and only if }x\leq y\text{ for all }x\in A\text{ and }y\in B.
\]
Instead of $\{a\}\leq B$ and $A\leq\{b\}$ we also write $a\leq B$ and $A\leq b$, respectively. Analogous notations are used for the reverse order $\geq$. Moreover, we define
\begin{align*}
L(A) & :=\{x\in P\mid x\leq A\}, \\
U(A) & :=\{x\in P\mid A\leq x\}.
\end{align*}
Instead of $L(A\cup B)$, $L(\{a\}\cup B)$, $L(A\cup\{b\})$ and $L(\{a,b\})$ we also write $L(A,B)$, $L(a,B)$, $L(A,b)$ and $L(a,b)$, respectively. Analogous notations are used for $U$. Instead of $L(U(A))$ we also write $LU(A)$. Analogously, we proceed in similar cases. Sometimes we identify singletons with their unique element, so we often write $L(a,b)=0$ and $U(a,b)=1$ instead of $L(a,b)=\{0\}$ and $U(a,b)=\{1\}$, respectively. The {\em poset} $\mathbf P$ is called {\em downward directed} if $L(x,y)\neq\emptyset$ for all $x,y\in P$. Of course, every poset with $0$ is downward directed. The {\em poset} $\mathbf P$ is called {\em bounded} if it has a least element $0$ and a greatest element $1$. This fact will be expressed by notation $(P,\leq,0,1)$.

The following concept was introduced in \cite{LR}: The poset $\mathbf P$ is called {\em modular} if
\begin{enumerate}
\item[(1)] $x\leq z$ implies $L(U(x,y),z)=LU(x,L(y,z))$.
\end{enumerate}
This is equivalent to
\[
x\leq z\text{ implies }UL(U(x,y),z)=U(x,L(y,z)).
\]
Recall from \cite{CL19} that $\mathbf P$ is called {\em strongly modular} if it satisfies the LU-identities
\begin{enumerate}
\item[(2)] $L(U(x,y),U(x,z))\approx LU(x,L(y,U(x,z)))$,
\item[(3)] $L(U(L(x,z),y),z)\approx LU(L(x,z),L(y,z))$.
\end{enumerate}
These are equivalent to
\begin{align*}
UL(U(x,y),U(x,z)) & \approx U(x,L(y,U(x,z))), \\
UL(U(L(x,z),y),z) & \approx U(L(x,z),L(y,z)),
\end{align*}
respectively. Observe that in case $x\leq z$ both (2) and (3) yield (1). Hence, every strongly modular poset is modular. Moreover, every modular lattice is a strongly modular poset. A strongly modular poset which is not a lattice is presented in Example~\ref{ex1}.

The {\em poset} $\mathbf P$ is called {\em distributive} if it satisfies the following identity:
\begin{enumerate}
\item[(4)] $L(U(x,y),z)\approx LU(L(x,z),L(y,z))$.
\end{enumerate}
This identity is equivalent to every single one of the following identities (see \cite{LR}):
\begin{align*}
UL(U(x,y),z) & \approx U(L(x,z),L(y,z)), \\
 U(L(x,y),z) & \approx UL(U(x,z),U(y,z)), \\
LU(L(x,y),z) & \approx L(U(x,z),U(y,z)).
\end{align*}
In fact, the inclusions
\begin{align*}
LU(L(x,z),L(y,z)) & \subseteq L(U(x,y),z), \\
UL(U(x,z),U(y,z)) & \subseteq U(L(x,y),z)
\end{align*}
hold in every poset. Hence, to check distributivity, we need only to confirm one of the converse inclusions. Observe that in case $x\leq z$ (4) implies (1). Hence every distributive poset is modular. Distributivity does not imply strong modularity. A unary operation $'$ on $P$ is called
\begin{itemize}
\item {\em antitone} if, for all $x,y\in P$, $x\leq y$ implies $y'\leq x'$,
\item an {\em involution} if it satisfies the identity $x''\approx x$,
\item a {\em complementation} if $L(x,x')\approx0$ and $U(x,x')\approx1$.
\end{itemize}
A {\em poset} is called {\em Boolean} if it is distributive and has a unary operation which is a complementation. For $A\subseteq P$ we define
\begin{align*}
\max A & :=\text{ set of all maximal elements of }A, \\
\max A & :=\text{ set of all minimal elements of }A, \\
    A' & :=\{x'\mid x\in A\}.
\end{align*} If the poset is bounded and distributive, we can prove the following property of an antitone involution.

\begin{lemma}
Let $(P,\leq,{}',0,1)$ be a bounded distributive poset with an antitone involution and $a,b\in P$ with $a\leq b$ and $L(b,a')=\{0\}$. Then the following hold:
\begin{align*}
L(a,a') & =L(b,b')=\{0\}, \\
U(a,a') & =U(b,b')=\{1\}.
\end{align*}
\end{lemma}

\begin{proof}
We have
\begin{align*}
L(a,a') & =LUL(a,a')=LU(L(a,a'),0)=LU(L(a,a'),L(b,a'))=L(U(a,b),a')= \\
        & =L(U(b),a')=L(b,a')=\{0\}, \\
L(b,b') & =LUL(b',b)=LU(0,L(b',b))=LU(L(a',b),L(b',b))=L(U(a',b'),b)= \\
        & =L(U(a'),b)=L(a',b)=\{0\}, \\
U(a,a') & =(L(a',a))'=\{0\}'=\{1\}, \\
U(b,b') & =(L(b',b))'=\{0\}'=\{1\}.
\end{align*}
\end{proof}

Now we recall the concept of a commutative meet-directoid from \cite{JQ}, see also \cite{CL11} for details. We will use it for the characterization of consistent posets which will be introduced below. The advantage of this approach is that we characterize properties of posets by means of identities and quasiidentities of algebras. Hence, one can use algebraic tools for their investigation.

A {\em commutative meet-directoid} (see \cite{CL11} and \cite{JQ}) is a groupoid $\mathbf D=(D,\sqcap)$ satisfying the following identities:
\begin{align*}
                   x\sqcap x & \approx x\text{ (idempotency)}, \\
                   x\sqcap y & \approx y\sqcap x\text{ (commutativity)}, \\
(x\sqcap(y\sqcap z))\sqcap z & \approx x\sqcap(y\sqcap z)\text{ (weak associativity)}.
\end{align*}
Let $\mathbf P=(P,\leq)$ be a downward directed poset. Define $x\sqcap y:=x\wedge y$ for comparable $x,y\in P$ and let $x\sqcap y=y\sqcap x$ be an arbitrary element of $L(x,y)$ if $x,y\in P$ are incomparable. Then $\mathbb D(\mathbf P):=(P,\sqcap)$ is a commutative meet-directoid which is called a {\em meet-directoid assigned} to $\mathbf P$. Conversely, if $\mathbf D=(D,\sqcap)$ is a commutative meet-directoid and we define for all $x,y\in D$
\begin{enumerate}
\item[(5)] $x\leq y$ if and only if $x\sqcap y=x$
\end{enumerate}
then $\mathbb P(\mathbf D):=(D,\leq)$ is a downward directed poset, the so-called {\em poset induced} by $\mathbf D$. Though the assignment $\mathbf P\mapsto\mathbb D(\mathbf P)$ is not unique, we have $\mathbb P(\mathbb D(\mathbf P))=\mathbf P$ for every downward directed poset $\mathbf P$. Sometimes we consider posets and commutative meet-directoids together with a unary operation. Let $(D,\sqcap,{}')$ be a commutative meet-directoid $(D,\sqcap,{}')$ with an antitone involution, i.e.\ $'$ is antitone with respect to the partial order relation induced by (5). We define
\[
x\sqcup y:=(x'\sqcap y')'\text{ for all }x,y\in D.
\]
Then $\sqcup$ is also idempotent, commutative and weakly associative and we have for all $x,y\in D$
\begin{align*}
x\sqcup y & =x\vee y\text{ if }x,y\text{ are }comparable, \\
x\sqcup y & =y\sqcup x\in U(x,y)\text{ if }x\parallel y, \\
x\sqcap y & =x\text{ if and only if }x\sqcup y=y, \\
     L(x) & =\{z\sqcap x\mid z\in P\}, \\
     U(x) & =\{z\sqcup x\mid z\in P\}, \\
   L(x,y) & =\{(z\sqcap x)\sqcap(z\sqcap y)\mid z\in P\}, \\
   U(x,y) & =\{(z\sqcup x)\sqcap(z\sqcup y)\mid z\in P\}.
\end{align*}

Posets with an antitone involution can be characterized in the language of commutative meet-directoids by identities as follows. The following lemma was proved in \cite{CKL1}. For the convenience of the reader we provide a proof.

\begin{lemma}\label{lem1}
Let $\mathbf P=(P,\leq,{}')$ be a downward directed poset with a unary operation and $\mathbb D(\mathbf P)$ an assigned meet-directoid. Then $\mathbf P$ is a poset with an antitone involution if and only if $\mathbb D(\mathbf P)$ satisfies the identities
\begin{enumerate}
\item[{\rm(6)}] $x''\approx x$,
\item[{\rm(7)}] $(x\sqcap y)'\sqcap y'\approx y'$.
\end{enumerate}
\end{lemma}

\begin{proof}
Condition (6) is evident by definition. Let $a,b\in P$. If (7) holds and $a\leq b$ then $b'=(a\sqcap b)'\sqcap b'=a'\sqcap b'\leq a'$ which shows that $'$ is antitone. If, conversely, $'$ is antitone then from $a\sqcap b\leq b$ we obtain $b'\leq(a\sqcap b)'$, i.e.\ $(a\sqcap b)'\sqcap b'=b'$ which is (7).
\end{proof}

\section{Characterizations by commutative meet-directoids}

Now we define our key concept.

\begin{definition}
A {\em consistent poset} is a bounded poset $(P,\leq,{}',0,1)$ with an antitone involution satisfying the following two conditions:
\begin{enumerate}
\item[{\rm(8)}] $L(x,x')=L(y,y')$ for all $x,y\in P\setminus\{0,1\}$,
\item[{\rm(9)}] $L(x,y)\neq0$ for all $x,y\in P\setminus\{0\}$.
\end{enumerate}
\end{definition}

It is easy to see that an at least three-element bounded poset $\mathbf P=(P,\leq,{}',0,1)$ with an antitone involution is consistent if and only if $\mathbf P$ has exactly one atom $a$ such that $P=[a,a']\cup\{0,1\}$ and $'$ is a complementation on the interval $([a,a'],\leq)$.

\begin{lemma}
The conditions {\rm(8)} and {\rm(9)} are independent.
\end{lemma}

\begin{proof}
The four-element Boolean algebra satisfies (8) but not (9), and the five-element chain (together with its unique possible antitone involution) satisfies (9) but not (8).
\end{proof}

In the following we list examples of consistent posets.

\begin{example}
The poset depicted in Figure~1
\vspace*{-2mm}
\begin{center}
\setlength{\unitlength}{7mm}
\begin{picture}(8,12)
\put(4,1){\circle*{.3}}
\put(4,3){\circle*{.3}}
\put(1,5){\circle*{.3}}
\put(3,5){\circle*{.3}}
\put(5,5){\circle*{.3}}
\put(7,5){\circle*{.3}}
\put(1,7){\circle*{.3}}
\put(3,7){\circle*{.3}}
\put(5,7){\circle*{.3}}
\put(7,7){\circle*{.3}}
\put(4,9){\circle*{.3}}
\put(4,11){\circle*{.3}}
\put(4,3){\line(0,-1)2}
\put(4,3){\line(-3,2)3}
\put(4,3){\line(-1,2)1}
\put(4,3){\line(1,2)1}
\put(4,3){\line(3,2)3}
\put(4,9){\line(-3,-2)3}
\put(4,9){\line(-1,-2)1}
\put(4,9){\line(1,-2)1}
\put(4,9){\line(3,-2)3}
\put(4,9){\line(0,1)2}
\put(1,5){\line(0,1)2}
\put(1,5){\line(1,1)2}
\put(3,5){\line(-1,1)2}
\put(3,5){\line(0,1)2}
\put(5,5){\line(0,1)2}
\put(5,5){\line(1,1)2}
\put(7,5){\line(-1,1)2}
\put(7,5){\line(0,1)2}
\put(3.85,.25){$0$}
\put(4.4,2.85){$a$}
\put(.3,4.85){$b$}
\put(3.4,4.85){$c$}
\put(4.3,4.85){$d$}
\put(7.4,4.85){$e$}
\put(.3,6.85){$e'$}
\put(3.4,6.85){$d'$}
\put(4.3,6.85){$c'$}
\put(7.4,6.85){$b'$}
\put(4.4,8.85){$a'$}
\put(3.3,11.4){$1=0'$}
\put(3.2,-.75){{\rm Fig.\ 1}}
\end{picture}
\end{center}

\vspace*{4mm}

is consistent, but neither modular since
\[
L(U(b,d),e')=L(a',e')=L(e')\neq L(b)=LU(b)=LU(a,b)=LU(b,L(d,e')),
\]
nor a lattice since $d'$ and $e'$ are different minimal upper bounds of $b$ and $c$.
\end{example}

\begin{example}\label{ex1}
The poset visualized in Figure~2
\vspace*{-2mm}
\begin{center}
\setlength{\unitlength}{7mm}
\begin{picture}(8,12)
\put(4,1){\circle*{.3}}
\put(4,3){\circle*{.3}}
\put(1,5){\circle*{.3}}
\put(3,5){\circle*{.3}}
\put(5,5){\circle*{.3}}
\put(7,5){\circle*{.3}}
\put(1,7){\circle*{.3}}
\put(3,7){\circle*{.3}}
\put(5,7){\circle*{.3}}
\put(7,7){\circle*{.3}}
\put(4,9){\circle*{.3}}
\put(4,11){\circle*{.3}}
\put(4,3){\line(0,-1)2}
\put(4,3){\line(-3,2)3}
\put(4,3){\line(-1,2)1}
\put(4,3){\line(1,2)1}
\put(4,3){\line(3,2)3}
\put(4,9){\line(-3,-2)3}
\put(4,9){\line(-1,-2)1}
\put(4,9){\line(1,-2)1}
\put(4,9){\line(3,-2)3}
\put(4,9){\line(0,1)2}
\put(1,5){\line(0,1)2}
\put(1,5){\line(1,1)2}
\put(1,5){\line(2,1)4}
\put(3,5){\line(-1,1)2}
\put(3,5){\line(2,1)4}
\put(3,5){\line(0,1)2}
\put(5,5){\line(-2,1)4}
\put(5,5){\line(1,1)2}
\put(5,5){\line(0,1)2}
\put(7,5){\line(-2,1)4}
\put(7,5){\line(-1,1)2}
\put(7,5){\line(0,1)2}
\put(3.85,.25){$0$}
\put(4.4,2.85){$a$}
\put(.3,4.85){$b$}
\put(2.3,4.85){$c$}
\put(5.4,4.85){$d$}
\put(7.4,4.85){$e$}
\put(.3,6.85){$e'$}
\put(2.3,6.85){$d'$}
\put(5.4,6.85){$c'$}
\put(7.4,6.85){$b'$}
\put(4.4,8.85){$a'$}
\put(3.3,11.4){$1=0'$}
\put(3.2,-.75){{\rm Fig.\ 2}}
\end{picture}
\end{center}

\vspace*{4mm}

is consistent and strongly modular, but not a lattice since $b'$ and $e'$ are different minimal upper bounds of $c$ and $d$.
\end{example}

\begin{example}
The poset depicted in Figure~3
\vspace*{-2mm}
\begin{center}
\setlength{\unitlength}{7mm}
\begin{picture}(8,14)
\put(4,1){\circle*{.3}}
\put(4,3){\circle*{.3}}
\put(1,5){\circle*{.3}}
\put(3,5){\circle*{.3}}
\put(5,5){\circle*{.3}}
\put(7,5){\circle*{.3}}
\put(1,7){\circle*{.3}}
\put(7,7){\circle*{.3}}
\put(1,9){\circle*{.3}}
\put(3,9){\circle*{.3}}
\put(5,9){\circle*{.3}}
\put(7,9){\circle*{.3}}
\put(4,11){\circle*{.3}}
\put(4,13){\circle*{.3}}
\put(4,3){\line(0,-1)2}
\put(4,3){\line(-3,2)3}
\put(4,3){\line(-1,2)1}
\put(4,3){\line(1,2)1}
\put(4,3){\line(3,2)3}
\put(4,11){\line(-3,-2)3}
\put(4,11){\line(-1,-2)1}
\put(4,11){\line(1,-2)1}
\put(4,11){\line(3,-2)3}
\put(1,5){\line(0,1)4}
\put(7,5){\line(0,1)4}
\put(1,7){\line(1,1)2}
\put(1,5){\line(1,1)4}
\put(3,5){\line(1,1)4}
\put(5,5){\line(1,1)2}
\put(3,5){\line(-1,1)2}
\put(5,5){\line(-1,1)4}
\put(7,5){\line(-1,1)4}
\put(7,7){\line(-1,1)2}
\put(4,11){\line(0,1)2}
\put(3.85,.25){$0$}
\put(4.4,2.85){$a$}
\put(.3,4.85){$b$}
\put(2.3,4.85){$c$}
\put(5.4,4.85){$d$}
\put(7.4,4.85){$e$}
\put(.3,6.85){$f$}
\put(7.4,6.85){$f'$}
\put(2.3,8.85){$d'$}
\put(.3,8.85){$e'$}
\put(5.4,8.85){$c'$}
\put(7.4,8.85){$b'$}
\put(4.4,10.85){$a'$}
\put(3.3,13.4){$1=0'$}
\put(3.2,-.75){{\rm Fig.\ 3}}
\end{picture}
\end{center}

\vspace*{4mm}

is consistent and distributive, but neither Boolean since $L(a,a')=a\neq0$, nor a lattice since $c'$ and $d'$ are different minimal bounds of $b$ and $e$.
\end{example}

Using the language of commutative meet-directoids, we can easily characterize lower cones $L(a,b)$ as follows.

\begin{lemma}\label{lem2}
Let $(P,\leq)$ be a downward directed poset, $a,b,c\in P$ and $(P,\sqcap)$ an assigned meet-directoid. Then $c\in L(a,b)$ if and only if $c=(c\sqcap a)\sqcap(c\sqcap b)$.
\end{lemma}

\begin{proof}
If $c\in L(a,b)$ then $c=c\sqcap c=(c\sqcap a)\sqcap(c\sqcap b)$. If, conversely, $c=(c\sqcap a)\sqcap(c\sqcap b)$ then
\begin{align*}
c & \leq c\sqcap a\leq a, \\
c & \leq c\sqcap b\leq b
\end{align*}
and hence $c\in L(a,b)$.
\end{proof}

Now we characterize consistent posets by means of commutative meet-directoids.

\begin{theorem}
Let $\mathbf P=(P,\leq,{}',0,1)$ be a bounded poset with a unary operation and $\mathbb D(\mathbf P)$ an assigned meet-directoid. Then $\mathbf P$ is consistent if and only if $\mathbb D(\mathbf P)$ satisfies identities {\rm(6)} and {\rm(7)} and implications {\rm(10)} and {\rm(11)}:
\begin{enumerate}
\item[{\rm(10)}] $x,y\neq0,1$ and $z=(z\sqcap x)\sqcap(z\sqcap x')$ imply $z=(z\sqcap y)\sqcap(z\sqcap y')$,
\item[{\rm(11)}] if $z=(z\sqcap x)\sqcap(z\sqcap y)$ implies $z=0$ then $x=0$ or $y=0$.
\end{enumerate}
\end{theorem}

\begin{proof}
\
\begin{enumerate}
\item[(10)] According to Lemma~\ref{lem2} the following are equivalent:
\begin{align*}
& (10), \\
& \text{if }x,y\neq0,1\text{ and }z\in L(x,x')\text{ then }z\in L(y,y'), \\
& \text{if }x,y\neq0,1\text{ then }L(x,x')\subseteq L(y,y'), \\
& \text{if }x,y\neq0,1\text{ then }L(x,x')=L(y,y').
\end{align*}
\item[(11)] According to Lemma~\ref{lem2} the following are equivalent:
\begin{align*}
& (11), \\
& \text{if }x,y\neq0\text{ then there exists some }z\neq0\text{ with }z\in L(x,y), \\
& \text{if }x,y\neq0\text{ then }L(x,y)\neq0.
\end{align*}
\end{enumerate}
Lemma~\ref{lem1} completes the proof.
\end{proof}

We can also characterize downward directed distributive posets in a similar manner. The following theorem was proved in \cite{CL2}. For the convenience of the reader we provide a proof.

\begin{theorem}
Let $\mathbf P=(P,\leq)$ be a downward directed poset and $\mathbb D(\mathbf P)$ an assigned meet-directoid. Then $\mathbf P$ is distributive if and only if $\mathbb D(\mathbf P)$ satisfies implication {\rm(12)}:
\begin{enumerate}
\item[{\rm(12)}] $w\sqcap((t\sqcup x)\sqcup(t\sqcup y))=w\sqcap z=w$ and $s\sqcup((t\sqcap x)\sqcap(t\sqcap z))=s\sqcup((t\sqcap y)\sqcap(t\sqcap z))=s$ for all $t\in P$ imply $w\leq s$.
\end{enumerate}
\end{theorem}

\begin{proof}
Since
\begin{align*}
& U(x,y)=\{(t\sqcup x)\sqcup(t\sqcup y)\mid t\in P\}, \\
& w\sqcap u=w\text{ is equivalent to }w\in L(u),
\end{align*}
$w\sqcap((t\sqcup x)\sqcup(t\sqcup y))=w\sqcap z=w$ is equivalent to $w\in L(U(x,y),z)$. Further, since
\begin{align*}
& L(x,z)=\{(t\sqcap x)\sqcap(t\sqcap z)\mid t\in P\}, \\
& L(y,z)=\{(t\sqcap y)\sqcap(t\sqcap z)\mid t\in P\}, \\
& s\sqcup u=s\text{ is equivalent to }s\in U(u),
\end{align*}
$s\sqcup((t\sqcap x)\sqcap(t\sqcap z))=s\sqcup((t\sqcap y)\sqcap(t\sqcap z))=s$ is equivalent to $s\in U(L(x,z),L(y,z))$. Hence the following are equivalent:
\begin{align*}
& (12), \\
& w\in L(U(x,y),z)\text{ and }s\in U(L(x,z),L(y,z))\text{ imply }w\leq s, \\
& L(U(x,y),z)\subseteq LU(L(x,z),L(y,z)), \\
& \mathbf P\text{ is distributive}.
\end{align*}
\end{proof}

\section{Residuation in consistent posets}

\begin{definition}
A {\em consistent residuated poset} is an ordered six-tuple $(P,\leq,\odot,\rightarrow,0,1)$ where $(P,\leq,0,1)$ is a bounded poset and $\odot$ and $\rightarrow$ are mappings {\rm(}so-called operators{\rm)} from $P^2$ to $2^P$ satisfying the following conditions for all $x,y,z\in P$:
\begin{itemize}
\item $x\odot y\approx y\odot x$,
\item $x\odot1\approx1\odot x\approx x$,
\item $x\odot y\leq z$ if and only if $x\leq y\rightarrow z$ {\rm(}adjointness{\rm)}.
\end{itemize}
\end{definition}

Let $(P,\leq,{}',0,1)$ be a poset with an antitone involution. Define mappings $\odot$ and $\rightarrow$ from $P^2$ to $2^P$ as follows:
\begin{enumerate}
\item[(13)] $\quad x\odot y:=\left\{
\begin{array}{ll}
0           & \text{if }x\leq y', \\
\max L(x,y) & \text{otherwise}
\end{array}
\right.
\quad\quad\quad x\rightarrow y:=\left\{
\begin{array}{ll}
1            & \text{if }x\leq y, \\
\min U(x',y) & \text{otherwise}
\end{array}
\right.$
\end{enumerate}

\begin{theorem}
Let $(P,\leq,{}',0,1)$ be a finite distributive consistent poset and $\odot$ and $\rightarrow$ be defined by {\rm(13)}. Then $(P,\leq,\odot,\rightarrow,0,1)$ is a consistent residuated poset.
\end{theorem}

\begin{proof}
Let $a,b,c\in P$. Because $a\leq b'$ is equivalent to $b\leq a'$ and, moreover, $L(a,b)=L(b,a)$, $\odot$ is commutative. Further,
\begin{align*}
& \text{if }a=0\text{ then }a\odot1=0=a, \\
& \text{if }a\neq0\text{ then }a\odot1=\max L(a,1)=\max L(a)=a.
\end{align*}
By commutativity of $\odot$ we obtain $a\odot1=1\odot a=a$. We consider the following cases:
\begin{itemize}
\item $a\leq b'$ and $b\leq c$. \\
Then $a\odot b=0\leq c$ and $a\leq1=b\rightarrow c$.
\item $a\leq b'$ and $b\not\leq c$. \\
Then $a\odot b=0\leq c$ and $a\leq b'\leq\min U(b',c)=b\rightarrow c$.
\item $a\not\leq b'$ and $b\leq c$. \\
Then $a\odot b=\max L(a,b)\leq b\leq c$ and $a\leq1=b\rightarrow c$.
\item $a\not\leq b'$, $b\not\leq c$. \\
In case $a=1$, $a\odot b\leq c$ and $a\leq b\rightarrow c$ are not possible because $a\odot b=1\odot b=b\not\leq c$. Moreover, $b,c'\neq0$ and therefore $b\rightarrow c=\min U(b',c)=(\max L(b,c'))'\neq0'=1$ whence $a=1\not\leq b\rightarrow c$. \\
Similarly, in case $c=0$, $a\odot b\leq c$ and $a\leq b\rightarrow c$ are not possible because $a,b\neq0$ and therefore $a\odot b=\max L(a,b)\neq0$ whence $a\odot b\not\leq c$. Moreover, $a\not\leq b'=\min U(b')=\min U(b',c)=b\rightarrow c$. \\
In case $b=1$ the following are equivalent:
\begin{align*}
a\odot b & \leq c, \\
 a\odot1 & \leq c, \\
       a & \leq c, \\
       a & \leq\min U(c), \\
       a & \leq\min U(1',c), \\
       a & \leq1\rightarrow c, \\
       a & \leq b\rightarrow c.
\end{align*}
There remains the case $a,b\neq1$ and $c\neq0$.  Then $a,b,c\neq0,1$. If $a\odot b\leq c$ then $\max L(a,b)\leq c$ and hence $L(a,b)\leq c$ whence
\begin{align*}
b\rightarrow c & =\min U(b',c)\subseteq U(b',c)\subseteq U(b',a\odot b)=U(b',L(a,b))= \\
               & =UL(U(b',a),U(b',b))=UL(U(b',a),U(a',a))\subseteq ULU(a)=U(a)
\end{align*}
which implies $a\leq b\rightarrow c$. If, conversely, $a\leq b\rightarrow c$ then $a\leq\min U(b',c)$ and hence $a\leq U(b',c)$ whence
\begin{align*}
a\odot b & =\max L(a,b)\subseteq L(a,b)\subseteq L(b\rightarrow c,b)=L(U(b',c),b)= \\
         & =LU(L(b',b),L(c,b))=LU(L(c',c),L(c,b))\subseteq LUL(c)=L(c)
\end{align*}
and hence $a\odot b\leq c$.
\end{itemize}
This shows that in any case $a\odot b\leq c$ is equivalent to $a\leq b\rightarrow c$.
\end{proof}

We now study residuation in not necessarily distributive consistent posets. For this purpose, we slightly modify our definition of residuation by deleting the assumption of commutativity of $\odot$.

\begin{definition}
A {\em weak consistent residuated poset} is an ordered six-tuple $(P,\leq,\odot,\rightarrow,0,1)$ where $(P,\leq,0,1)$ is a bounded poset and $\odot$ and $\rightarrow$ are mappings {\rm(}so-called operators{\rm)} from $P^2$ to $2^P$ satisfying the following conditions for all $x,y,z\in P$:
\begin{itemize}
\item $x\odot1\approx1\odot x\approx x$,
\item $x\odot y\leq z$ if and only if $x\leq y\rightarrow z$ {\rm(}adjointness{\rm)}.
\end{itemize}
\end{definition}

Let $(P,\leq,{}',0,1)$ be a poset with an antitone involution. We modify the definition of the mappings (so-called operators) $\odot$ and $\rightarrow$ from $P^2$ to $2^P$ in the following way:
\begin{enumerate}
\item[(14)]
$x\odot y:=\left\{
\begin{array}{ll}
0                 & \text{if }x\leq y', \\
\max L(U(x,y'),y) & \text{otherwise}
\end{array}
\right.
x\rightarrow y:=\left\{
\begin{array}{ll}
1                 & \text{if }x\leq y, \\
\min U(x',L(x,y)) & \text{otherwise}
\end{array}
\right.$
\end{enumerate}

Now, we are able to prove our second result on residuation.

\begin{theorem}
Let $(P,\leq,{}',0,1)$ be a finite strongly modular consistent poset and $\odot$ and $\rightarrow$ be defined by {\rm(14)}. Then $(P,\leq,\odot,\rightarrow,0,1)$ is a weak consistent residuated poset.
\end{theorem}

\begin{proof}
Let $a,b,c\in P$. If $a=0$ then $a\odot1=0=a$ and $1\odot a=0=a$. If $a\neq0$ then
\begin{align*}
 a\odot1 & =\max L(U(a,1'),1)=\max LU(a)=\max L(a)=a, \\
1\odot a & =\max L(U(1,a'),a)=\max L(a)=a.
\end{align*}
We consider the following cases:
\begin{itemize}
\item $a\leq b'$ and $b\leq c$. \\
Then $a\odot b=0\leq c$ and $a\leq1=b\rightarrow c$.
\item $a\leq b'$ and $b\not\leq c$. \\
Then $a\odot b=0\leq c$ and $a\leq b'\leq\min U(b',L(b,c))=b\rightarrow c$.
\item $a\not\leq b'$ and $b\leq c$. \\
Then $a\odot b=\max L(U(a,b'),b)\leq b\leq c$ and $a\leq1=b\rightarrow c$.
\item $a\not\leq b'$, $b\not\leq c$. \\
In case $a=1$, $a\odot b\leq c$ and $a\leq b\rightarrow c$ are not possible because $a\odot b=1\odot b=b\not\leq c$. Moreover, $b,c'\neq0$ and hence $L(b,c')\neq0$ which implies $L(b,U(b',c'))\neq0$ and therefore $b\rightarrow c=\min U(b',L(b,c))=(\max L(b,U(b',c')))'\neq0'=1$ whence $a=1\not\leq b\rightarrow c$. \\
Similarly, in case $c=0$, $a\odot b\leq c$ and $a\leq b\rightarrow c$ are not possible because $a,b\neq0$ and hence $L(a,b)\neq0$ whence $L(U(a,b'),b)\neq0$ and therefore $a\odot b=\max L(U(a,b'),b)\neq0$ whence $a\odot b\not\leq c$. Moreover, $a\not\leq b'=\min U(b')=\min U(b',L(b,c))=b\rightarrow c$. \\
In case $b=1$ the following are equivalent:
\begin{align*}
a\odot b & \leq c, \\
 a\odot1 & \leq c, \\
       a & \leq c, \\
       a & \leq\min U(c), \\
       a & \leq\min UL(c), \\
       a & \leq\min U(1',L(1,c)), \\ 
       a & \leq1\rightarrow c, \\
       a & \leq b\rightarrow c.
\end{align*}
There remains the case $a,b\neq1$ and $c\neq0$.  Then $a,b,c\neq0,1$. If $a\odot b\leq c$ then
\begin{align*}
b\rightarrow c & =\min U(b',L(b,c))\subseteq U(b',L(b,c))\subseteq U(b',L(b,a\odot b))= \\
               & =U(b',L(b,\max L(U(a,b'),b)))=U(b',L(b)\cap L(\max L(U(a,b'),b)))= \\
               & =U(b',L(b)\cap L(U(a,b'),b))=U(b',L(b,U(a,b')))=UL(U(b',b),U(a,b'))= \\
               & =UL(U(a',a),U(a,b'))\subseteq ULU(a)=U(a)
\end{align*}
which implies $a\leq b\rightarrow c$. If, conversely, $a\leq b\rightarrow c$ then
\begin{align*}
a\odot b & =\max L(U(a,b'),b)\subseteq L(U(a,b'),b)\subseteq L(U(b\rightarrow c,b'),b)= \\
         & =L(U(\min U(b',L(b,c)),b'),b)=L(U(\min U(b',L(b,c)))\cap U(b'),b)= \\
         & =L(U(b',L(b,c))\cap U(b'),b)=L(U(b',L(b,c)),b)=L(U(L(b,c),b'),b)= \\
         & =LU(L(b,c),L(b',b))=LU(L(b,c),L(c',c))\subseteq LUL(c)=L(c)
\end{align*}
and hence $a\odot b\leq c$.
\end{itemize}
This shows that in any case $a\odot b\leq c$ is equivalent to $a\leq b\rightarrow c$.
\end{proof}

\section{Dedekind-MacNeille completion}

In what follows we investigate the question for which posets $\mathbf P$ with an antitone involution their Dedekind-MacNeille completion $\BDM(\mathbf P)$ is a consistent lattice. A bounded {\em lattice} $(L,\vee,\wedge,{}',0,1)$ with an antitone involution is called {\em consistent} if it is consistent when considered as a poset, i.e.\ if
\begin{align*}
x\wedge x' & =y\wedge y'\text{ for all }x,y\in L\setminus\{0,1\}, \\
 x\wedge y & \neq0\text{ for all }x,y\in L\setminus\{0\}.
\end{align*}
Let $\mathbf P=(P,\leq,{}')$ be a poset with an antitone involution. Define
\begin{align*}
 \DM(\mathbf P) & :=\{L(A)\mid A\subseteq P\}, \\
            A^* & :=L(A')\text{ for all }A\in\DM(\mathbf P), \\
\BDM(\mathbf P) & :=(\DM(\mathbf P),\subseteq,^*)
\end{align*}
Then $\BDM(\mathbf P)$ is a complete lattice with an antitone involution, called the {\em Dedekind-MacNeille completion} of $\mathbf P$. That $^*$ is an antitone involution on $(\DM(\mathbf P),\subseteq)$ can be seen as follows. Let $A,B\in\DM(\mathbf P)$. If $A\subseteq B$ then $A'\subseteq B'$ and hence $B^*=L(B')\subseteq L(A')=A^*$. Moreover, $A^{**}=L((L(A'))')=LU(A)=A$. We have
\begin{align*}
 (L(A))^* & =L((L(A))')=LU(A')\text{ for all }A\subseteq P, \\
  A\vee B & =LU(A,B)\text{ for all }A,B\in\DM(\mathbf P), \\
A\wedge B & =A\cap B\text{ for all }A,B\in\DM(\mathbf P).
\end{align*}

\begin{theorem}\label{th6}
Let $\mathbf P=(P,\leq,{}')$ be a poset with an antitone involution. Then $\BDM(\mathbf P)$ is a consistent lattice if and only if $\mathbf P$ is a consistent poset.
\end{theorem}

\begin{proof}
Assume $\mathbf P$ to be a consistent poset. Further assume $A\subseteq P$ and $L(A)\neq0,P$. Then $1\notin L(A)$ and there exists some $a\in L(A)\setminus\{0\}$. Hence $0\notin U(A')$ and $a'\in U(A')\setminus\{1\}$. Now
\begin{align*}
L(A)\wedge(L(A))^* & =L(A)\cap LU(A')=\bigcup_{x\in L(A)}L(x)\cap\bigcap_{y\in U(A')}L(y)= \\
                   & =\bigcup_{x\in L(A)\setminus\{0\}}L(x)\cap\bigcap_{y\in U(A')\setminus\{1\}}L(y).
\end{align*}
Now
\begin{align*}
L(a,a') & =\bigcap_{y\in U(A')\setminus\{1\}}L(y',y)=\bigcap_{y\in U(A')\setminus\{1\}}(L(y')\cap L(y))\subseteq \\
        & \subseteq\bigcap_{y\in U(A')\setminus\{1\}}(\bigcup_{x\in L(A)\setminus\{0\}}L(x)\cap L(y))=\bigcup_{x\in L(A)\setminus\{0\}}L(x)\cap\bigcap_{y\in U(A')\setminus\{1\}}L(y)= \\
        & =\bigcup_{x\in L(A)\setminus\{0\}}(L(x)\cap\bigcap_{y\in U(A')\setminus\{1\}}L(y))\subseteq\bigcup_{x\in L(A)\setminus\{0\}}(L(x)\cap L(x'))= \\
        & =\bigcup_{x\in L(A)\setminus\{0\}}L(x,x')=L(a,a')
\end{align*}
and hence $L(A)\wedge(L(A))^*=L(a,a')$. This shows $L(A)\wedge(L(A))^*=L(B)\wedge(L(B))^*$ for all $A,B\subseteq P$ with $L(A),L(B)\neq0,P$. Now assume $A,B\subseteq P$ and $L(A),L(B)\neq0$. Then there exists some $a\in L(A)\setminus\{0\}$ and some $b\in L(B)\setminus\{0\}$. Since $\mathbf P$ is consistent there exists some $c\in L(a,b)\setminus\{0\}$. Now $L(c)\subseteq L(a)\subseteq L(A)$, $L(c)\subseteq L(b)\subseteq L(B)$ and $0\neq c\in L(c)$ and hence $L(c)\neq0$. This shows that $\BDM(\mathbf P)$ is a consistent lattice provided $\mathbf P$ is a consistent poset. The converse is evident.
\end{proof}

Compliance with Ethical Standards: This study was funded by the Austrian Science Fund (FWF), project I~4579-N, and the Czech Science Foundation (GA\v CR), project 20-09869L, as well as by \"OAD, project CZ~02/2019, and, concerning the first author, by IGA, project P\v rF~2020~014. The authors declare that they have no conflict of interest. This article does not contain any studies with human participants or animals performed by any of the authors.

Authors' addresses:

Ivan Chajda \\
Palack\'y University Olomouc \\
Faculty of Science \\
Department of Algebra and Geometry \\
17.\ listopadu 12 \\
771 46 Olomouc \\
Czech Republic \\
ivan.chajda@upol.cz

Helmut L\"anger \\
TU Wien \\
Faculty of Mathematics and Geoinformation \\
Institute of Discrete Mathematics and Geometry \\
Wiedner Hauptstra\ss e 8-10 \\
1040 Vienna \\
Austria, and \\
Palack\'y University Olomouc \\
Faculty of Science \\
Department of Algebra and Geometry \\
17.\ listopadu 12 \\
771 46 Olomouc \\
Czech Republic \\
helmut.laenger@tuwien.ac.at
\end{document}